\documentclass[a4paper,12pt]{amsart}
\usepackage[ngerman,english]{babel}
\usepackage[ansinew]{inputenc}
\usepackage{amssymb}
\usepackage{amsmath}
\usepackage{amsfonts}
\usepackage{amstext}
\usepackage{amsthm}
\usepackage{graphics}
\usepackage{graphicx}
\usepackage[T1]{fontenc}
\usepackage{textcomp}

\usepackage[font=footnotesize]{caption}

\usepackage{enumerate}
\usepackage{hyperref}

\newcommand{\Hmm}[1]{\leavevmode{\marginpar{\tiny%
			$\hbox to 0mm{\hspace*{-0.5mm}$\leftarrow$\hss}%
			\vcenter{\vrule depth 0.1mm height 0.1mm width \the\marginparwidth}%
			\hbox to
			0mm{\hss$\rightarrow$\hspace*{-0.5mm}}$\\\relax\raggedright #1}}}

\newcommand{\IN}{{\mathbb{N}}}
\newcommand{\IR}{{\mathbb{R}}}

\newcommand{\R}{{\mathbb{R}}}

\newcommand{\1}{\mathbbmss{1}}

\renewcommand{\phi}{\varphi}
\newcommand{\ph}{\varphi}
\renewcommand{\epsilon}{\varepsilon}
\newcommand{\al}{\alpha}

\renewcommand{\ge}{>}

\newcommand{\Qgen}{{\widetilde{Q}}}
\newcommand{\Dgen}{{\widetilde{D}}}

\newcommand{\ltwo}{{\ell^2(X,m)}}
\newcommand{\lp}{{\ell^p(X,m)}}

\newcommand{\NeuDom}{{D(Q)}}
\newcommand{\NeuDir}{{Q}}

\newcommand{\NeuLap}{{L}}

\newtheorem{theorem}{Theorem}[section]
\newtheorem{lemma}[theorem]{Lemma}
\newtheorem{proposition}[theorem]{Proposition}

\theoremstyle{definition}
\newtheorem{remarks}[theorem]{Remark}
\newtheorem{definition}[theorem]{Definition}

\begin{document}
\title[The Kazdan-Warner equation on graphs]
{The Kazdan-Warner equation on  canonically compactifiable  graphs}
\author[M.~Keller]{Matthias Keller}
\address{M.~Keller,   Institut f\"ur Mathematik, Universit\"at Potsdam
	\\14476  Potsdam, Germany}
\email{{matthias.keller@uni-potsdam.de}}
\author[M.~Schwarz]{Michael Schwarz}
\address{M.~Schwarz,  Institut f\"ur Mathematik, Universit\"at Potsdam
	\\14476  Potsdam, Germany}
\email{mschwarz@math.uni-potsdam.de}


\begin{abstract}
	We study the Kazdan-Warner equation on canonically compactifiable  graphs. These graphs are distinguished as analytic properties of Laplacians on these graphs carry a strong resemblance to Laplacians on open pre-compact manifolds. 
\end{abstract}

\maketitle
\section{Introduction}
In recent years various topics in the analysis on graphs received a lot of attention. While the main focus was put so far mainly on the study of linear equations such as the heat equation, the Poisson equation or fundamental topics in spectral theory, see e.g. \cite{BHLLMY,BHY, BKW,DLMSY,Fol3,Gol,HJL,HuaLin,HuK,KL1,KePiPo2,Mue14,Woj1,Web}, there is an upcoming interest in non-linear problems. On the one  hand this concerns non-linear operators such as the $ p $-Laplacian, \cite{HuMu15,KeMu16,Mug13}, and on the other hand non-linear equations such as the Kazdan-Warner equation \cite{GLY16,Ge17}.

A major goal is often to unveil the analogies in the continuum, i.e., Riemannian manifolds, and the discrete setting, i.e. graphs. The deeper reason of these analogies stems from the fact that both the Laplacian in the continuum and in the discrete are generators of so called Dirichlet forms or from another perspective they are generators of certain Markov processes, see \cite{Fuk}. A major challenge in the discrete is often to find the correct geometric notions to even formulate analogous statements.

The Kazdan-Warner equation arises from the basic geometric question which functions are potential curvatures of $ 2 $-dimensional manifolds. Kazdan and Warner studied this question  on compact manifolds \cite{KaWa74a} and on open manifolds that are diffeomorphic to an open set in a compact  manifold \cite{KaWa74b}. From there on there is an enormous amount of work on this topic and we refer here only to \cite{ChLi93,DJLW97}

Compactness in the discrete setting is obviously equivalent to finiteness of the graph. For this case a very satisfying answer was found  by Grigor'yan/Lin/Yang \cite{GLY16}, see also \cite{Ge17}. On the other hand it is not a priori clear what are graph analogues to  pre-compact open manifolds. This is the starting point of our paper. We propose a class called canonically compactifiable graphs. These infinite graphs were recently introduced in \cite{GHKLW}, see also \cite{KLSW17}, and the analysis of Laplacians on  these graphs resembles a lot of features from the analysis of Laplacians on bounded domains or pre-compact open manifolds with sufficiently nice boundaries. For example the Dirichlet problem on the Royden boundary is uniquely solvable and the spectrum of the Laplacian is purely discrete in the case of finite measure. Here, for the Kazdan-Warner equation, we prove analogous results as in the finite case for canonically compactifiable graphs, and, indeed, our results include finite graphs as a special case.

Let us recall the original geometric motivation of the Kazdan-Warner equation.  This traces back to an equation on a $ 2 $-dimensional Riemannian manifold $ M $ with finite volume and two conformal metrics $ g  $ and  $ \widetilde g=e^{2\ph}g $ for some smooth function $ \ph $. Specifically denoting the corresponding Gaussian curvatures by $ K $ and $ \widetilde K $ this gives rise to the equation
\begin{align*}
\Delta_{g}\ph=K-\widetilde K e^{2\ph},
\end{align*}
with the Laplace-Beltrami operator $ \Delta_{g} $ with respect to the metric $ g $. Provided there is a solution $ \psi $ to the equation $ \Delta_{g} \psi=K-\overline K$ with $ \overline K $ being the averaged curvature and letting $ u=2(\ph-\psi) $ this translates into the equation
\begin{align*}
\Delta_{g} u=2\overline K-(2\widetilde Ke^{2\psi})e^{u}.
\end{align*}
Hence, the question whether a prescribed function $ \widetilde  K$ is a curvature function under a conformal transformation is equivalent to the question whether the equation above has a solution $ u $. Of course, from the geometric picture 	 one already expects certain restrictions on $ \widetilde{K} $ depending on $ \overline K $ as for example that  $ \widetilde{K} $ is positive somewhere when  $ \overline K>0 $ or that $ \widetilde{K} $ changes sign if $ \widetilde{K} \neq 0$ in the flat case $ \overline K =0$ (at least when $ M $ is compact).

Although there are already several very promising notions of curvatures proposed on graphs, see e.g. \cite{BHLLMY,BP1,BP2,EM2,Fo03,Mue14,KPP,Oll1,Oll2,Wi1}, it still seems to be out of reach to study this problem in terms of one of these curvature notions. However, looking at this equation from an analytic perspective  is interesting in its own right. This reduces to the study of an equation
\begin{align*}
Lu=-c+he^{u}
\end{align*}
where $ L $ is the graph Laplacian, $ c $ is a constant, $ h $ is a function, and $ u $ is a function in the domain of $ L $, see Section~\ref{section_graphs} for the precise definitions.

As stated above we study this equation on so called canonically compactifiable graphs and under the assumption that the graph has finite measure. One way of characterizing canonically compactifiable graphs is that functions of finite energy are already bounded, see \cite{GHKLW} (while for this work it is sufficent that all finite energy functions in $\ell^2$ are bounded).  

Our results can be  summarized as follows in dependence of the sign of $ c $ and where $ \overline h $ denotes the average of $ h $:
\begin{itemize}
	\item[$ c=0 :$] There exists a solution if and only if  $ \overline h<0 $ and $ h  $ changes sign.
	\item[$ c>0 :$] There exists a solution if and only if $ h  $ is positive somewhere.
	\item[$ c<0 :$] If there exists a solution then $ \overline h<0 $ and if $ \overline h<0 $ then there is $-\infty< c(h)<0 $ such that there is a solution for $ c(h)<c<0 $.
\end{itemize}
We only assume $ h \neq0$ to be square summable. Moreover,  we consider the so-called Neumann-Laplacian which is the self-adjoint operator associated to the quadratic form on all $ \ell^{2} $-functions of finite energy. Solutions are then requested to be in the  operator domain of the Neumann-Laplacian. So, we are indeed talking about strong solutions of the equation.  From this it is clear that while the basic ideas of our proofs owe to \cite{GLY16}, they  need substantially more care on the operator theoretic level. Furthermore, in contrast to the case of manifolds we do not have a theory of elliptic regularity in the discrete case and no chain rule.

The paper is structured as follows. In the next section we introduce the setting and in particular canonically compactifiable graphs. We furthermore provide basic tools such as a Sobolev embedding theorem, a Poincaré inequality and a Trudinger Moser inequality. In Section~\ref{section_c=0} we characterize existence of solutions in the case $ c=0 $ and in Section~\ref{section_c>0} we study the case $c>0$. Finally, in Section~\ref{section_c<0} we give a necessary and sufficient criterion on the existence of solutions in the case $c<0$.

In this paper $ C $ denotes an arbitrary positive constant which might change from line to line. If we want to indicate that a constant $ C $ depends on a parameter $ \beta $ we write $ C=C(\beta) $.

\emph{Note added:} After this work was completed we learned about the recent preprint of Ge/Jiang \cite{GeJi17}. This paper studies the Kazdan-Warner equation on infinite graphs, however with different geometric assumptions on the graph and on the function $ h $. It seems that these assumptions are independent from the assumptions of this paper and the techniques   used there are totally different from ours.

\section{The Laplacian on canonically compactifiable graphs}\label{section_graphs}

In this section we introduce canonically compactifiable graphs over discrete measure spaces. These graphs, which can be seen as discrete analogues to open pre-compact manifolds,  were studied systematically in \cite{GHKLW}, see also \cite{KLSW15}. The definition is here slightly weakened by taking also a finite measure into consideration. We then introduce several tools that are needed for the proof. The availability of these tools such as the Sobolev embedding theorem, a Poincaré inequality and a Trudinger Moser inequality, which hold trivially in the finite setting, see \cite{GLY16},  owes here to the assumption on the graph being canonically compactifiable and of finite measure.

Let $X$ be a countable\footnote{Here, \emph{countable} means finite or countably infinite, although we are of course mainly interested in the countably infinite case} set. A symmetric function $b:X\times X\to [0,\infty)$ that vanishes on the diagonal and satisfies 
\begin{align*}
\sum_{y\in X} b(x,y)<\infty
\end{align*}
for every $x\in X$ is called a \emph{graph} on $X$.

A graph is called \emph{connected} if for every $x,y\in X$ there are $x_1,\ldots, x_n\in X$ such that ${b(x,x_1)>0, b(x_1,x_2)>0,\ldots, b(x_n,y)>0}$.
We equip the set $X$ with the discrete topology, so that every measure $m$ on (the Borel-$\sigma$-field of)  $X$ is given by a  function, also denoted by $m$, from $X$ to $[0,\infty]$, via
$m(A):=\sum_{x\in A} m(x)$. Throughout this paper we assume that  $m(x)>0$ for all $ x\in X $ and that $ m $ is a finite measure, 
\begin{align*}
m(X)<\infty,
\end{align*}
and call   $(X,m)$ a \emph{finite discrete measure space}.
Given a measure $m$ on $X$ we define $\lp$ as the space of real valued functions on $X$ which are $ p $-square-summable with respect to $m$, $ 1\leq p<\infty $ and $ \ell^{\infty}(X)  $ the space of bounded functions. Denote the corresponding norm $ \|\cdot\|_p $ by 
$$ \|f\|_p:=
\left(\sum_{x\in X}|f(x)|^{p}m(x)\right)^{1/p} $$ Furthermore, we denote the dual pairing of $\ell^p$ and $\ell^q$ for $ p,q\in[1,\infty] $ with $ 1/p +1/q=1$ by  $\langle \cdot,\cdot\rangle$.

Denote by $C(X)$ the space of real valued functions. We define a quadratic form $\Qgen:C(X)\to [0,\infty]$ via
\[\Qgen(u)=\frac12\sum_{x,y\in X} b(x,y)(u(x)-u(y))^2.\]
Let $$ \Dgen:=\{u\in C(X)\mid \Qgen(u)<\infty\}. $$ 

Via polarization $\Qgen$ gives rise to a semi-inner product on $ \Dgen $ also denoted by $\Qgen$. It is easy to see that $ \Qgen $ is \emph{Markovian}, i.e.
\[\Qgen(0\vee u\wedge 1)\leq\Qgen(u), \qquad u\in\Dgen .\] 
By Fatou's lemma $ \widetilde{Q} $ is lower semi-continuous and, therefore,
 the restriction $\NeuDir$ of $\Qgen$ to $$ \NeuDom:=\Dgen\cap\ltwo $$ 
is closed and, hence, a Dirichlet form which is often referred to as the \emph{Dirichlet form with Neumann boundary conditions} or the \emph{Neumann form}. We denote the \emph{form norm} by
\begin{align*}
\|\cdot\|_{\NeuDir}:=\left(\NeuDir(\cdot)+\|\cdot\|_{2}^{2}\right)^{1/2}.
\end{align*}
Furthermore, we let $ \NeuDir(f):=\infty $ for $ f\in \ell^{2}(X,m)\setminus D(\NeuDir) $.

There is a self-adjoint operator, the \emph{Neumann Laplacian}, $L$ in correspondence with $\NeuDir$ given by 
\begin{align*}D(\NeuLap)=\{u\in\NeuDom\mid& \text{there is } f\in\ltwo \text{ such that }\\&\NeuDir(u,v)=\langle f,v\rangle \text{ holds for every } v\in\NeuDom\},\end{align*}
\[\NeuLap u=f.\]
From this definition it is clear that the constant functions are in $ D(L) $ in the case of finite measure.
It turns out, \cite{HKLW},  by the virtue of Green's formula $ \NeuLap $ acts as
\[\NeuLap u(x)=\frac{1}{m(x)}\sum_{y\in X} b(x,y)(u(x)-u(y)),\qquad  u\in D(\NeuLap).\]

The goal of this paper is to study the so called \emph{Kazdan-Warner equation} for functions $ u\in D(\NeuLap) $
\[\NeuLap u=-c+he^u,\] 
for some given function $ h\in\ell^{2}(X,m) $, $ h\neq 0 $, and a constant $ c $.
We will study this equation on what we call \emph{canonically compactifiable graphs}
over $ (X,m) $.

\begin{definition}
 A connected graph $b$ over a finite discrete measure space $(X,m)$ is called \emph{canonically compactifiable} if $$ \NeuDom\subseteq \ell^\infty(X). $$ 
\end{definition}

\begin{remarks}
	The notion of canonically compactifiable graphs $ b $ over a discrete set $ X $ (instead of a  measure space $ (X,m) $) was introduced in \cite{GHKLW} via the condition
	\begin{align*}
	\Dgen\subseteq \ell^{\infty}(X).
	\end{align*}
	For these graphs the Royden compactification (which arises as space of maximal ideals of the commutiative Banach-algebra of $ \Dgen $ closed with respect to $ \|\cdot\|_{\infty} $) turns out to have particular nice properties. For example the Dirichlet problem with respect to the Royden boundary is uniquely solvable. Furthermore,  various results in  \cite{GHKLW} suggest that canonically compactifiable stand in parallel to bounded domains in $ \R^{d} $. 
\end{remarks}

Next, we present two embedding results.

\begin{proposition}\label{lemma_embedding_Neumann_supnorm}
	Let $b$ be a connected, canonically compactifiable graph over $(X,m)$ with $m(X)<\infty$. Then, there is a $C>0$ such that \[\|u\|_\infty\leq C\|u\|_\NeuDir,\qquad u\in \NeuDom.\]
\end{proposition}
\begin{proof}
	The statement follows from the closed graph theorem.	
\end{proof}

\begin{proposition}[Sobolev Embedding, Theorem~5.1,~Corollary~5.2 \cite{GHKLW}]\label{lemma_Sobolev}
 Let $b$ be a connected, canonically compactifiable graph over $(X,m)$ with $ m(X)<\infty $. Then,
  $(\NeuDom,\|\cdot\|_\NeuDir)$ embeds compactly into $\ltwo$. In particular, the operator $\NeuLap$ has purely discrete spectrum.
\end{proposition}
\begin{proof}
	For the semigroup  $e^{-t\NeuLap}  $, $ t\geq0 $, of $ \NeuLap $  we have by the definition of canonically compactifiable graphs and finiteness of the measure $ m $
	\begin{align*}
	e^{-tL}\ell^{2}(X,m)\subseteq D(\NeuLap)\subseteq D(\NeuDir)\subseteq \ell^{\infty}(X)\subseteq \ell^{2}(X,m).
	\end{align*} 
	Thus, $ e^{-t\NeuLap} $  can be viewed as an operator from $ \ell^{2}(X,m) $ to $ \ell^{\infty}(X) $ composed with the embedding $ \ell^{\infty}(X)\to\ell^{2}(X,m) $ and is, therefore, a Hilbert-Schmidt operator by the factorization principle, see \cite{Sto94}. However, compactness of $e^{-tL} $ is equivalent to the compactness of the embedding of $ D(\NeuDir) $ into $ \ell^{2}(X,m) $ which is equivalent to discreteness of the spectrum of $ \NeuLap $.
\end{proof}

\begin{proposition}\label{KerLN}
 Let $b$ be a connected, canonically compactifiable graph over $(X,m)$  with $ m(X)<\infty $. Then, $\ker(\NeuLap)=\operatorname{span}\{1\}$. Furthermore,  $ \NeuLap $  from $ \{1\} ^{\perp}$ to $ \{1\}^{\perp} $ is invertible.
\end{proposition}
\begin{proof}
	The first part of the statement carries over from \cite[Theorem~7.1]{GHKLW}. The second part now follows from the fact that by the Sobolov embedding above, $ \NeuLap $ has purely discrete spectrum and is, therefore, bijective on $ \{1\}^{\perp} $.
\end{proof}

In the case of finite measure $ m $ we obviously have
\begin{align*}
\ell^{\infty}(X)\subseteq\ell^{p}(X,m)\subseteq \ell^{1}(X,m),\qquad p\in[1,\infty),
\end{align*}
and, hence,  for $ u\in\ell^{p}(X,m) $ we can define
\begin{align*}
\overline u:=\frac{1}{m(X)}\langle u,1\rangle.
\end{align*}
Next, we come to a Poincar\'e-inequality on canonically compactifiable graphs.

\begin{proposition}[Poincar\'e-inequality]\label{proposition_Poincare}
 Let $b$ be a connected, canonically compactifiable graph over $(X,m)$  with $ m(X)<\infty $. Then, there is $C>0$ such that
 \[\|u-\overline u\|_2^2\leq C\NeuDir(u),\qquad u\in\NeuDom.\] 
\end{proposition}
\begin{proof}
Suppose  there is a sequence $(u_n)$ in $\NeuDom$ such that $ \overline{u_{n}}=0 $,  $\NeuDir(u_n)\to 0$ and $\|u_n\|_2=1$.
Since $ (u_{n}) $ is bounded with respect to $\|\cdot\|_\NeuDir$  it contains an $\ell^{2}$-convergent subsequence by the Sobolev embedding, Proposition~\ref{lemma_Sobolev}. Without loss of generality we assume
 that $(u_n)$ itself is $\ell^{2}$-convergent to a limit  $u$. Obviously, $\|u\|_2=1$ and due to the finiteness of $m$, we have $ \overline u =0$ since
 $ \ell^{2} $-convergence induces $ \ell^{1} $-convergence (by the closed graph theorem).
 Since $\NeuDir$ is lower semicontinuous with respect to $\ltwo$-convergence, we conclude \[\NeuDir(u)\leq\liminf_{n\to\infty} \NeuDir(u_n)=0\] and, therefore $u$ must be constant. Thus, $ \overline{u}=0 $ implies $ u=0 $ which however contradicts $\|u\|_2=1$. Hence the  inequality holds for every $ u \in \NeuDom$ with $ \overline u=0 $. For general $ u\in \NeuDom $ observe that $ \overline{u-\overline u}=0 $ and $\NeuDir( u-\overline u) =\NeuDir( u)$.
 \end{proof}

Next, we discuss the following Trudinger-Moser inequality.
\begin{proposition}[Trudinger-Moser inequality]\label{proposition_Trudinger}
 Let $b$ be a connected, canonically compactifiable graph over  $(X,m)$ with $ m(X)<\infty $. Then, for every $\beta\in\IR$ there is a constant $C>0$  such that
 \[\langle e^{\beta (u-\overline{u})^2},1\rangle\leq m(X)e^{C|\beta|\NeuDir(u)} \]
 for all $  u\in\NeuDom $.
 In particular, for all $ \beta\in\R $ there is $ C=C(\beta) $ such that 
 \begin{align*}
 \langle e^{\beta u^{2}},1\rangle\leq C(\beta)
 \end{align*}
 for all $ u\in\NeuDom $ with $ \NeuDir(u)=1 $ and $ \overline{u}=0 $.
\end{proposition}
\begin{proof}
 The inequality holds for $\beta\leq 0$ since $m(X)<\infty$. Let $\beta>0$. By Proposition~\ref{lemma_embedding_Neumann_supnorm} and the Poincar\'e  inequality, Proposition~\ref{proposition_Poincare}, $$(u(x)-\overline{u})^2\leq\|u-\overline{u}\|_\infty^2\leq C\|u-\overline u\|_\NeuDir^{2}
 =C(\NeuDir(u)+\|u-\overline u\|_2)\leq C \NeuDir(u).$$
 Hence, the statement follows.  
\end{proof}

Next, we show that for canonically compactifiable graphs with finite measure the form domain is invariant under composition with continuously differentiable functions.
\begin{lemma}\label{lemma_diff_in_dom}
 Let $b$ be a connected, canonically compactifiable graph over $(X,m)$ with $ m(X)<\infty $. Let $u\in \NeuDom$  and $f:\IR\to\IR$ be continuously differentiable. Then, $f\circ u\in \NeuDom$.
\end{lemma}
\begin{proof}
Let $u\in  \NeuDom \subseteq\ell^{\infty}(X) $. Since both $f$ and $f'$ are continuous, we get that $f\circ u,f'\circ u\in\ell^{\infty}(X)\subseteq \ell^{2}(X)$. 
Let $ C:= \|f'|_{[-\|u\|_\infty,\|u\|_\infty]}\|_\infty^2$.
For  any
 $x,y\in X$ we infer 
 by the mean value theorem
 that there exists a point $\xi$ between $u(x)$ and $u(y)$ such that
 \[(f\circ u(x)-f\circ u(y))^2=f'(\xi)^2(u(x)-u(y))^2<C(u(x)-u(y))^2\]  
 Thus, $\Qgen(f\circ u)\leq C \Qgen(u)<\infty$ and, hence, $f\circ u\in \NeuDom$.
\end{proof}

\section{The case $c=0$}\label{section_c=0}

\begin{theorem}[The case $c=0$]\label{theorem_case_c_equal_zero}
	Let $b$ be a canonically compactifiable graph over $(X,m)$ such that $ m(X)<\infty $ and $0\neq h\in \ell^2(X,m)$. Then, there is an $u\in D(\NeuLap)$ with $$\NeuLap u=he^u$$ if and only if $\overline h<0$ and $h$ 
	changes sign. 
\end{theorem}
\begin{proof}
 Suppose $u\in D(\NeuLap)$ is a solution. Then, $ u\in\ell^{\infty}(X) $ since the graph is graph canonically compactifiable and, hence,    $he^{u}\in\ell^{2}(X,m)  $. Furthermore, $ m(X)<\infty $ also implies 
 $1\in \NeuDom$ and as $ u  $ is a solution, we obtain
  \[\langle he^{u},1\rangle=\langle\NeuLap u,1\rangle=\NeuDir(u,1)=0.\] 
 Hence, $h$ must change sign, as $h\not\equiv 0$ by assumption. Then,  since $e^{-u}\in \NeuDom$ by Lemma~\ref{lemma_diff_in_dom}, we  compute 
 \begin{align*}
 \langle h,1\rangle=\langle he^{u},e^{-u}\rangle=\langle \NeuLap u, e^{-u}\rangle=\NeuDir(u,e^{-u})<0,
 \end{align*}
 where the last inequality follows by definition of $ \NeuDir $ and the fact that $ t\mapsto e^{-t} $ is monotone decreasing. Furthermore,  the inequality is strict since $u$ can not be a constant if $h$ changes its sign, or otherwise we would have $0=Lu=he^u$ and, therefore, $h\equiv 0$ which contradicts our assumption.
 
 Let us turn to the other direction. So, assume $h\in\ell^2(X,m)$ is such that  $\overline h<0$ and there is $ x_{0}\in X $ such that $$ h(x_0)>0. $$ 
 Define the set 
 \[B=\{v\in \NeuDom\mid \langle he^{v},1\rangle= \langle v,1\rangle=0\}.\] 
 The strategy of the proof is to show that $ Q $ admits a minimizer on $ B $ which turns out to be a solution up to an additive constant. As a first step we show that $ B $ is non empty.\\

 \textbf{Claim 1:} The set $B$ is not empty.
 \begin{proof}[Proof of Claim 1] Define $v_0=1_{x_0}$. Then, $tv_0\in \NeuDom$ for every $t\in\IR$ and a simple calculation shows
 	\[\langle he^{tv_0},1\rangle=(e^t-1)h(x_0)m(x_0)+\langle h,1\rangle.\] Thus, $$F:[0,\infty)\to \IR,\quad t\mapsto \langle he^{tv_0},1\rangle$$ is a continuous map with 
 	$F(0)<0$ and $F(t)>0$ for large $t$. Hence, there is $t_0$ such that $F(t_0)=0$. Then, $v=t_0v_0-\frac{t_0m(x_0)}{ m(X)}\in\NeuDom$ and $v$ satisfies \[\langle v,1\rangle=0,\] and 
 	\[\langle he^{v},1\rangle =
 	e^{-\frac{t_0m(x_0) }{m(X)}}\langle he^{t_{0}v_{0}},1\rangle=e^{-\frac{t_0 m(x_0)}{m(X)}}F(t_0)=0.\] Thus, we get $v\in B$. 
 \end{proof}

 \textbf{Claim 2:} There is a minimizer of $ \NeuDir $ in  $ B$.
 \begin{proof}[Proof of Claim 2]
 	Let $(w_n)$ be a sequence in $B$ such that $\NeuDir(w_n)\to \inf_{w\in B}\NeuDir(w)$. Using the Poincar\'e 
 	inequality, Proposition~\ref{proposition_Poincare}, we infer that $(w_n)$ is a bounded sequence with respect to $\|\cdot\|_\NeuDir$. Hence, $ (w_{n}) $ has an $\ltwo$-convergent subsequence by the Sobolev Embedding, Proposition~\ref{lemma_Sobolev}, and we assume without loss of generality  that 
 	$(w_n)$ itself is $\ltwo$-convergent.  We denote the limit by $w$.
 We infer $w\in D(\NeuDir)$ by  the lower semi-continuity of $\NeuDir$, i.e.,
 	\[\NeuDir(w)\leq \liminf_{n\to\infty}\NeuDir(w_n)=\inf_{w\in B}\NeuDir(w).\] 
 	 	Next we show $w\in B$.  By Proposition~\ref{lemma_Sobolev} we have $\sup_{n\in\IN}\|w_n\|_\infty<\infty$. Hence, by the finiteness of $m$ and Lebesgue's dominated convergence theorem we observe \[\langle w,1\rangle=\lim_{n\to\infty}\langle w_n,1\rangle=0.\] Furthermore, we have 
 	$|he^{w}-he^{w_n}| \leq |h|(e^{\|w\|_\infty}+e^{\sup_{n\in\IN}\|w_n\|_\infty})$ and another application of Lebesgue's  dominated convergence theorem yields
 	$$\langle he^{w},1\rangle=\lim_{n\to\infty}  \langle he^{w_n},1\rangle=0.$$
 	This shows $w\in B$ and, therefore, $ \inf_{w\in B}\NeuDir(w)\leq \NeuDir(w)$. This finishes the proof of the claim.
 \end{proof}
 Let $w$ be a minimizer of $\NeuDir$ on $B$.
 Then, the map $\NeuDom\to\IR$, $f\mapsto \NeuDir(f)$ has a minimum for $f=w$ under the restrictions 
 \[\langle he^{f},1\rangle=0\qquad \mbox{and} \qquad\langle f,1\rangle=0.\] 
 We  apply the  Lagrange multiplier theorem  in Banach spaces, c.f. \cite[Theorem 43.D]{zeidler}.
 For this, we calculate the Fréchet derivatives of the maps 
 \begin{align*}
 Q&:D(\NeuDir)\to\R, \quad f\mapsto \NeuDir(f)\\
  F&:D(\NeuDir)\to\R,\quad  f\mapsto\langle f,1\rangle\\
  H&:D(\NeuDir)\to\R,\quad  f\mapsto\langle he^{f},1\rangle.
 \end{align*}
 
 \textbf{Claim 3:} The maps $ Q$ , $F$  and  $H $ are continuously Fréchet differentiable with Fréchet derivatives at $ f\in D(\NeuDir) $
 given by
 \begin{align*}
 D Q[f](\cdot)  =2\NeuDir(f,\cdot),\quad D F[f](\cdot)= \langle \cdot ,1\rangle,\quad D H[f](\cdot)=\langle \cdot he^{f},1\rangle.
 \end{align*}
 \begin{proof}[Proof of Claim 3]
 	For the derivative of $Q$ we calculate  
 	\begin{align*}
 	\frac{|\NeuDir(f+g)-\NeuDir(f)-2\NeuDir(f,g)|}{\|g\|_\NeuDir}
 	&=\frac{\NeuDir(g)}{\|g\|_\NeuDir}\leq\frac{\|g\|_\NeuDir^2}{\|g\|_\NeuDir}\to 0,
 	\end{align*}  
 	for $\|g\|_\NeuDir\to 0$. By the Cauchy-Schwarz-inequality the  Fréchet derivative $ DQ[f](\cdot) =2Q(f,\cdot)$ is continuous.
 	
The statement for $F$ follows  easily as $ F $ is linear.

 	Finally, for the derivative of $H$ we estimate using that 
 	$$\frac{\|e^g-1-g\|_\infty}{\|g\|_{\NeuDir}}\leq C\frac{\sum_{k=2}^\infty \frac{\|g\|_\infty^k}{k!}}{\|g\|_\infty}
 	\leq C\sum_{k=1}^\infty\frac{\|g\|_\infty^k}{k!}=
 	 C( e^{\|g\|_{\infty}}-1)\to 0$$ for  $ \|g\|_{\NeuDir}\to0 $	 
 	  since $\|\cdot\|_\NeuDir$-convergence implies $\|\cdot\|_\infty$-convergence by Proposition~\ref{lemma_Sobolev}, and 
 	 \begin{align*}
 	 \frac{|\langle he^{f+g},1\rangle-\langle he^f,1\rangle-\langle hge^f,1\rangle|}{\|g\|_\NeuDir}
 	 =&\frac{|\langle he^{f}(e^g-1-g),1\rangle|}{\|g\|_\NeuDir}
 	  	\to0
 	 \end{align*}
 	 	To see that the Fr\'echet-derivative of $ H $ is continuous, we estimate by the Cauchy-Schwarz-inequality and $ \|\cdot\|_{2}\leq \|\cdot\|_{\NeuDir} $
 	\begin{align*}
 |DH[f-g](\ph)|=	|\langle \varphi h(e^f-e^g),1\rangle|&=|\langle h(e^f-e^g),\varphi\rangle|\\
 	&\leq \|h(e^f-e^g)\|_2\|\varphi\|_\NeuDir
 	\end{align*}
 This proves the continuity of the Fréchet derivative of $H$.
 \end{proof}
 Let $ w $ be a minimizer of $ Q $ on $ B $ whose existence was proven above.
 The map $$(DF[w],DH[w]): D(\NeuDir)\to\IR^2,\quad f\mapsto (\langle f,1\rangle,\langle f he^{w},1\rangle) $$ is surjective, since there are $x_0,x_1\in X$ with $h(x_0)>0, h(x_1)<0$ by assumption, and 
 $1_{x_0},1_{x_1}\in \NeuDom$, and, thus, ${DF[w](1_{x_0}), DF[w](1_{x_1})>0}$ and $DH[w](1_{x_0})>0, DH[w](1_{x_0})<0$. 
 
  Thus, we can apply the Lagrange multiplier theorem in Banach spaces and infer that there are $\lambda,\mu\in\IR$ such that
 \[2\NeuDir(w,g)=\lambda \langle ghe^{w},1\rangle+\mu \langle g,1\rangle=
  \langle \lambda he^{w} +\mu,g\rangle\] holds for every $g\in\NeuDom$, where $\lambda he^{w}+\mu\in\ltwo$ as $m$ is finite and $w\in D(\NeuDir)$ is 
 bounded, since the graph is canonically compactifiable. Therefore, by definition, $w\in D(\NeuLap)$ and $$ \NeuLap w=\frac\lambda2 he^{w}+\frac\mu2. $$
 
 Furthermore, by $1\in \NeuDom$ we deduce \[0=2\NeuDir(w, 1)=\lambda\langle he^{w},1\rangle +\mu m(X).\]
 Since $w\in B$ we see $\langle he^{w},1\rangle=0$ and, thus, $\mu=0$.
 
 Finally, we show $\lambda>0$. It is easy to see that $\lambda$ can not be zero, since otherwise we would infer $w\equiv 0$ and the constant zero is not in $B$.
 Thus, we have $0>2\NeuDir(w,e^{-w})=\lambda\langle h,1\rangle$, where the first inequality can be seen by an easy calculation. Since by assumption $\langle h,1\rangle=\overline h m(X)<0$, we infer $\lambda>0$. 
 Therefore, we get $\frac\lambda2 =e^\sigma$ for some $\sigma\in\R$ and the function $u=w+\sigma\in D(\NeuLap)$ is a solution. This concludes the proof.
\end{proof}

\section{The case $c>0$}\label{section_c>0}
\begin{theorem}[The case $c>0$]
	Let $b$ be a canonically compactifiable graph over $(X,m)$ such that $ m(X)<\infty $. Let $c>0$ and $0\neq h\in \ell^2(X,m)$. Then, there is an $u\in D(\NeuLap)$ with \[\NeuLap u=he^u-c\] if and only if $h$ is positive somewhere.  
\end{theorem}
\begin{proof}
	Suppose $u\in D(\NeuLap)$ is a solution. Then, by \[0=\NeuDir(u,1)=\langle \NeuLap u, 1 \rangle=\langle -c,1 \rangle+\langle he^u,1 \rangle\] we infer \[cm(X)=\langle he^{u},1 \rangle\] and, since 
	$c>0$, the function $h$ must be positive somewhere.
	
	Now suppose that $h$ is positive somewhere. Define \[B=\{v\in \NeuDom\mid \langle he^{v}-c,1 \rangle =0\}\] 
		and 
		$$J:D(\NeuDir)\to\R,\qquad		
		J(v):=\frac12\NeuDir(v)+cm(X)\overline{v}.$$
	The strategy of the proof is similar to the proof of Theorem~\ref{theorem_case_c_equal_zero} above. We show that $ J $ assumes a minimum on $ B $ which turns out to be a solution. To this end we first show that $ B $ is non-empty and $ J $ is bounded below on $ B $. Afterwards we show that $ J $ assumes a minimum and then we show that this minimum is a solution by the virtue of Lagrange multipliers.\\
	
	\textbf{Claim 1:} The set $B$ is not empty. 
	\begin{proof}[Proof of Claim 1]
		Let $x_0\in X$ be such that 
		$h(x_0)>0$ and define ${v_{t,l}=t1_{x_0}-l}$ for $t,l\in \IR$. 
		We get $v_{t,l}\in \NeuDom$ for every $t,l\in \IR$. Define $F:\IR^2\to\IR$,
		\[F(t,l):=\langle he^{v_{t,l}},1 \rangle=(e^t-1)e^{-l}h(x_0)m(x_0)+e^{-l}\langle h,1 \rangle\]
		which  is obviously continuous. Furthermore, for fixed 
		$t$ we infer $F(t,l)\to 0$ for $l\to\infty$ and for fixed $l$ we infer $F(t,l)\to \infty $ for $t\to\infty$. 
		Since the continuous image of connected sets is connected, we conclude that there are $t_0,l_0\in\IR$ such that $F(t_0,l_0)=cm(X)$ and, hence, $v_{t_0,l_0}\in B$. 
	\end{proof}
Next we show that $ J $ is bounded from below on $ B $.\\

	\textbf{Claim 2:} There is a constant $C$ such that 
	$J(v)\geq \frac14\NeuDir(v)+C$, $v\in B$.
	\begin{proof}[Proof of Claim 2]
		For $v\in B$ we infer
		$$  e^{\overline{v}}=\frac{cm(X)}{\langle he^{v-\overline v},1 \rangle}  $$
		from the calculation
		\begin{align*}
		\langle he^{v-\overline v},1 \rangle=e^{-\overline{v}}\langle he^{v},1 \rangle=e^{-\overline{v}}cm(X).
		\end{align*}
		
		Therefore, \begin{align*}
		J(v)&=\frac12\NeuDir(v)+cm(X)\overline{v}\\&=\frac12\NeuDir(v)+cm(X)\log(e^{\overline{v}})\\
		&=\frac12\NeuDir(v)+cm(X)\log(cm(X))-cm(X)\log(\langle he^{v-\overline v},1 \rangle).
		\end{align*}
		The idea is now to control the third term on the right hand side by $ \NeuDir(v) $.
		To this end we  apply the Cauchy-Schwarz inequality, the basic inequality 
		$2st\leq 2\varepsilon s^2+\frac{2}{\varepsilon}t^2$ for $ s=\NeuDir(v)^{1/2} $, $ t=(v-\overline v)/\NeuDir(v)^{1/2} $ and $ \varepsilon>0 $, and, the Trudinger-Moser inequality, Proposition~\ref{proposition_Trudinger} for the function $ u=(v-\overline v)/\NeuDir (v)={ (v-\overline v)}/\NeuDir(v-\overline v)$ and $ \beta=2/\varepsilon $
		\begin{align*}
		\langle he^{v-\overline v},1 \rangle^{2}\leq\|h\|_{2}^{2}		\langle e^{2(v-\overline v)},1 \rangle\leq \|h\|_{2}^{2}\langle e^{2\varepsilon\NeuDir(v)+\frac{2(v-\overline v)^2}{\varepsilon \NeuDir(v)}},1 \rangle
		\leq \|h\|_{2}^{2}Ce^{2\varepsilon\NeuDir(v)}.
		\end{align*}
		Therefore, by the computation for $ J(v) $ above we obtain with  $\varepsilon={1}/({4cm(X)})$
		\begin{align*}
		J(v)
		&\geq\frac12\NeuDir(v)+cm(X)\log(cm(X))-cm(X) \log(C\|h\|_2 e^{\varepsilon\NeuDir(v)})\\
		&= \frac14 \NeuDir(v)+cm(X)\log\left(\frac{cm(X)}{C\|h\|_2}\right)\\
		&= \frac14 \NeuDir(v)+C.
		\end{align*}
	\end{proof}
	
	\textbf{Claim 3:} There is a minimizer of $J$ in $B$.  
	\begin{proof}[Proof of Claim 3]
		Let $u_n$ be a minimizing sequence in $B$ such that $J(u_n)\to  \inf_{w\in B}J(w)>-\infty$. This yields by Claim~2 and by  the Poincar\'e inequality, Proposition~\ref{proposition_Poincare}, that the sequence $(u_k-\overline u_{k})$ is a bounded sequence in 
		$(\NeuDom,\|\cdot\|_\NeuDir)$. By the definition of $ J $ \[\overline{u}_k=\frac{1}{m(X)}(J(u_k)-\NeuDir(u_k))\] we infer that the sequence $(\overline{u}_k)$ is bounded in $\IR$. Hence, by 
		\[\|u_k\|_\NeuDir\leq \|u_k-\overline{u}_{k}\|_\NeuDir+\|\overline{u}_{k}\|_\NeuDir=\|u_k-\overline{u}_{k}\|_\NeuDir+\overline{u_k}m(X)^{\frac12}\] the sequence $(u_k)$ is bounded in $(\NeuDom,\|\cdot\|_\NeuDir)$.
		By  the weak compactness of closed balls in Hilbert spaces $ (u_{k}) $ has a $\|\cdot\|_\NeuDir$ weakly convergent subsequence  which is without loss of generality $ (u_{k}) $ itself. By the Sobolev Embedding, Proposition~\ref{lemma_Sobolev}, $ (u_{k}) $ converges to an $ u $ in $ \ell^{2}(X,m) $ which is in $ D(\NeuDir) $ as well.
		
		 Furthermore, by Proposition~\ref{lemma_Sobolev} we infer
		$\sup_{n\in\IN}\|u_k\|_\infty<\infty$ and, hence, $u$ is bounded as well. 	By the inequality $|he^{u}-he^{u_n}| \leq |h|(e^{\|u\|_\infty}+e^{\sup_{n\in\IN}\|u_n\|_\infty})$ we can apply Lebesgue's  dominated convergence theorem to obtain
		\[\langle he^{u},1 \rangle=\lim_{n\to\infty}  \langle he^{u_n},1 \rangle=cm(X).\]
		Therefore $u\in B$ and, by the lower semi-continuity of $\NeuDir$, we infer \[J(u)\leq \liminf_{n\to\infty}J(u_n)=\inf_{v\in B}J(v)\] and we conclude that $u$ is a minimizer.  
	\end{proof}
	Let $u\in B$ be a minimizer of $J$ on $B$.
	Then the map $J$ has a minimum for $u\in D(\NeuDir)$ under the restriction given by the map 	$ I: D(\NeuDir)\to \R  $
		\[I(f):=\langle he^{f}-c,1 \rangle=0.\]  
	So, our next goal is to apply the theorem about Lagrange multipliers in Banach spaces, c.f. \cite{zeidler}[Theorem 43.D]. For this we 
	need to calculate the Fréchet derivatives of $ J $ and $ I $.\\

	\textbf{Claim 4:} The maps $ J $ and $ I $  are continuously Fréchet differentiable and have Fréchet derivatives at $ f\in D(\NeuDir) $ given by
	\begin{align*}
	DJ[f](\cdot)=\NeuDir(f,\cdot)+c\langle\cdot,1 \rangle,\qquad DI[f](\cdot)=\langle \cdot he^{f},1 \rangle.
	\end{align*}
	\begin{proof}[Proof of Claim 4] With the notation from the proof  of Theorem~\ref{theorem_case_c_equal_zero} the maps $ J $ and $ I $ can be represented as
		\begin{align*}
		J=\frac12Q+cF\quad \mbox{and}\quad I= H-cF.
		\end{align*}
	Hence, the statement follows from Claim~3 of the proof of Theorem~\ref{theorem_case_c_equal_zero}.
	\end{proof}  
	It is easy to see that the map $DI[f](\cdot):\NeuDom\to\IR$ is surjective. Hence, we can apply \cite{zeidler}[Theorem 43.D] and infer the existence of $\lambda\in\IR$ such that   
	\[\NeuDir(u,f)+c\langle f,1 \rangle=\lambda \langle fhe^{u},1 \rangle\] holds for every $f\in\NeuDom$ and the minimizer $ u $. 

	Hence, \[\NeuDir(u,f)=\langle -c+\lambda he^{u},f \rangle\] for every $f\in \NeuDom$. 
	Note that $\lambda he^{u}-c\in\ltwo$ as $u$ is bounded and $m$ is finite.
	Therefore, by definition, $u\in D(\NeuLap)$ and $$ \NeuLap u=\lambda he^{u}-c. $$	
	By $1\in \NeuDom$ we infer \[0=\NeuDir(u, 1)=\lambda\langle he^{u},1 \rangle - cm(X).\]
	Since $u\in B$ we see $\langle he^{u},1 \rangle=cm(X)$ and, thus, $\lambda=1$.
	This concludes the proof.
\end{proof}

 \section{The case $ c<0 $}\label{section_c<0}
In this section we treat the case $ c<0 $ and prove the following theorem.

\begin{theorem}[The case $ c<0 $]\label{t:c<0}
	Let $b$ be a connected, canonically compactifiable graph over $(X,m)$ and $ m(X)<\infty $. Let  $0\neq h\in\ell^2(X,m)$ and  $c>0$. Then, the following holds:
	 \begin{itemize}
	 	\item [(1)] If there is $u\in D(\NeuLap)$ with $\NeuLap u=he^u-c$, then $\overline{h}<0$.
	 	\item[(2)]  If $\overline{h}<0$, then there exists a constant $-\infty \leq c_-(h)<0$ such that the equation $\NeuLap u=he^u-c$ has a solution for all $0>c>c_-(h)$ and no solution for every $c<c_-(h)$. If $ h\leq0 $, then $ c_{-}(h)=-\infty $.
	 \end{itemize}
\end{theorem}
While the proof of (1) follows by a direct calculation, the proof of (2) involves the construction of a solution via lower and upper solutions. In particular,
we call a function $u_0\in D(\NeuLap)$ a \emph{lower solution} if \[\NeuLap u_0\leq he^{u_0}-c,\]  and we call a function $u_1\in D(\NeuLap)$ an\emph{ upper solution} if 
\[\NeuLap u_1\geq he^{u_1}-c.\] 

We introduce for a function $ k:X\to (0,\infty) $ the bilinear form 
$ \NeuDir+k $   on $\ell^2(X,m)$ via the form sum
\begin{align*}
&D(\NeuDir+k)=D(\NeuDir)\cap \ell^{2}(X,km)\\
&(\NeuDir+k)(f,g)=\NeuDir(f,g)+\langle kf,g\rangle,\qquad f,g\in D(\NeuDir+k).
\end{align*}
An immediate consequence of Fatou's lemma (and the fact that $ D(\NeuDir)=\widetilde{D}\cap\ell^{2}(X,m) $) yields that $ \NeuDir+k $ is lower semicontiuous and, hence,  closed. We denote the associated non-negative, selfadjoint operator on $ \ell^{2}(X,m) $ by $ \NeuLap+k $ which acts as
\begin{align*}
(\NeuLap+k )f(x)=\NeuLap f(x)+ k(x)f(x),\qquad f\in D(\NeuLap+k), x\in X.
\end{align*}
We show next that for canonically compactifiable graphs over  $ (X,m) $ with $ m(X)<\infty $ the domains of the forms $  \NeuDir$ and $ \NeuDir+k $ and the corresponding operators coincide whenever $ k\in\ell^{2}(X,m) $.

 \begin{lemma}\label{lemma_L+k}
 	Let $ b $ be a  canonically compactifiable graph over $ (X,m) $, $ m(X)<\infty $ and $ k\in\ell^{2}(X,m) $ with $ k>0 $. Then,
 	\begin{align*}
 	D(\NeuLap+k)= D(\NeuLap).
 	\end{align*}
 	Moreover, the operator $L+k$ is bijective.
 \end{lemma}
\begin{proof}
Let $ f\in \NeuDom $. 
We estimate	by using Hölder's inequality and Proposition~\ref{lemma_Sobolev}
\begin{align*}
\langle kf,f\rangle \leq \|k\|_{1}\|f\|_{\infty}^{2}\leq C \|k\|_{1}\NeuDir(f).
\end{align*}
Since $ k\in\ell^{2}(X,m) $, we infer $\|k\|_{1}<\infty$ as  $\ell^{2}(X,m)\subseteq \ell^{1}(X,m) $ in the case of finite measure. Hence, $ f\in \ell^2(X,km) $ and, therefore,
\begin{align*}
D(\NeuDir)\subseteq\ell^{2}(X,km).
\end{align*}
We conclude
$$ D(\NeuDir+k)=D(\NeuDir)\cap \ell^{2}(X,km) =D(\NeuDir).$$	

To show the equality of the operator domains let $ f,g\in\ell^{2}(X,m) $. 
Let $ G_{0}=(\NeuLap+1)^{-1} $ and $ G_{k}=((\NeuLap+k)+1)^{-1} $. Furthermore, with slight abuse of notation we  denote the operator of multiplication by $ k $ on $ \ell^{2}(X,m) $ by $ k $.
Since $$  G_{k}\ell^{2}(X,m)\subseteq D(\NeuDir+k)=D( \NeuDir)\subseteq \ell^{\infty}(X)  $$
we have
\begin{align*}
k G_{k}f\in \ell^{2}(X,m)
\end{align*}
and we calculate
\begin{align*}
\langle G_{0}k G_{k}f,g\rangle=\langle k G_{k}f,G_{0}g\rangle= (\NeuDir+k)(G_{k}f,G_{0}g)-\NeuDir(G_{k}f,G_{0}g)
\end{align*}
Using the facts that $ (\NeuDir+k)(G_{k}f,G_{0}g)+\langle  G_{k}f,G_{0}g\rangle=\langle f,G_{0}g\rangle $ and  $ \NeuDir(G_{k}f,G_{0}g)+\langle  G_{k}f,G_{0}g\rangle=\langle G_{k}f,g\rangle $ we infer
\begin{align*}
\langle G_{0}k G_{k}f,g\rangle&=\langle G_{k}f,g\rangle -\langle f,G_{0}g\rangle = \langle (G_{k}-G_{0})f,g\rangle.
\end{align*}
Since $ g $ was chosen arbitrarily in $ \ell^{2}(X,m) $ we have for all $ f\in \ell^{2}(X,m) $
\begin{align*}
G_{k}f=G_{0}k G_{k}f +G_{0}f
\end{align*}
Since $ G_{k} $ is surjective on $ D(\NeuLap+k) $ and $ G_{0}\ell^{2}(X,m)\subseteq D(\NeuLap) $ we infer the statement $  	D(\NeuLap+k)\subseteq D(\NeuLap) $.

The  other inclusion $ D(\NeuLap)\subseteq D(\NeuLap+k)  $ follows analogously.

Finally, we show that $L+k$ is bijective. Note, that $L+k$ has compact resolvent by the equality $G_{k}f=G_{0}k G_{k}f +G_{0}f$, $f\in\ell^2(X,m)$, shown above and since $G_0$ is compact by Proposition~\ref{lemma_Sobolev}. It is left to show that $\ker(L+k)=\{0\}$. Let $u\in\ker(L+k)$. Then, we have
$$0=\langle (L+k)u,u\rangle=(Q+k)(u,u)\geq \langle k u,u\rangle$$ and, since $k> 0$, we infer $u\equiv 0$.  
\end{proof}

Furthermore, we need the following maximum principle
\begin{lemma}[Maximum principle]\label{MaxPrin}
	Let $b$ be a graph over $(X,m)$ and $ k\in\ell^{2}(X,m) $, $ k\ge0 $. Let $u\in D(\NeuLap)$, $u$ not constant, such that the inequality $(\NeuLap+k) u\leq 0$ holds. Then, we get $u\leq 0$. 
\end{lemma}
\begin{proof}
	By  $u_{\pm}=(\pm u)\vee 0\in D(\NeuDir+k)$ one has \[0\geq \langle (\NeuLap +k)u, u_+\rangle=(\NeuDir+k)(u,u_+)=(\NeuDir+k)(u_+,u_+)-(\NeuDir+k)(u_+,u_-).\] Hence,
	\[(\NeuDir+k)(u_+,u_+)\leq(\NeuDir+k)(u_+,u_-)\leq 0,\] where the latter inequality holds by $(\NeuDir+k)(|u|)\leq (\NeuDir+k)(u)$ as $k>0$ and $Q$ is a Dirichlet form. Thus $u_+\equiv 0$ since $k>0$ by assumption. 
\end{proof}

With these preparations we  show that the existence of suitable lower and upper solutions imply the existence of a solution.

\begin{lemma}\label{lemma_upper_lower_implies_solution}
	Let $b$ be a connected, canonically compactifiable graph over $(X,m)$, $ m(X)<\infty $ and let $c<0$, $0\not=h\in \ell^2(X)$.
	If there is an upper solution $u_1$ and a lower solution $u_0$ such that $u_1\geq u_0$, then there is a solution $u\in D(\NeuLap)$ with $u_1\geq u\geq u_0$ and $\NeuLap u=he^u-c$.
\end{lemma}
\begin{proof}
Taking the upper solution $ u_{1}\in D(\NeuLap) $ and $ h\in\ell^{2}(X) $, we let
	\begin{align*}
	k=(1\vee(-h))e^{u_1}.
	\end{align*}
Since the graph is canonically compactifiable, and thus $ D(\NeuLap)\subseteq \ell^{\infty}(X) $, we have $ e^{u_{1}}\in \ell^{\infty}(X) $ and, therefore, 
\begin{align*}
k\geq e^{\inf u_1}>0\qquad \mbox{and}\qquad k\in\ell^{2}(X,m).
\end{align*}
Hence, the operator  $\NeuLap+k$ is bijective by Lemma~\ref{lemma_L+k}. 
For $ u\in C(X) $ and $ c $ from our assumption we let the right hand side of the Kazdhan-Warner equation be denoted by
$$ f_u:X\to\IR,\qquad f_u=he^{u}-c .$$
If $ u\in \ell^{\infty}(X,m) $, in particular if $ u\in D(\NeuDir+k)=D(\NeuDir) $, we have $f_u\in\ltwo$. 
Starting with $ u_{1} $ we inductively define
\[u_{j+1}=(\NeuLap+k)^{-1}(f_{u_j}+ku_j)\in\ltwo,\qquad j\geq 1.\]

 We next show $u_0\leq \ldots\leq u_{j+1}\leq u_j\leq \ldots\leq u_1$ by induction. For $ j=2 $  we calculate 
 \[(\NeuLap+k) (u_2-u_1)= f_{u_1}+ku_1-\NeuLap u_1-ku_1= f_{u_1}-\NeuLap u_1\leq 0\] as $u_1$ is an upper solution.
	Using the maximum principle, Lemma~\ref{MaxPrin}, we infer $u_2-u_1\leq 0$. 	For $j\geq 2$  assume $u_j\leq u_{j-1}\leq\ldots\leq u_1$. We infer
	\begin{align*}
	(\NeuLap+k) (u_{j+1}-u_j)&=f_{u_j}-f_{u_{j-1}}+ku_j-ku_{j-1}\\
	&=h\cdot(e^{u_j}-e^{u_{j-1}})+k\cdot(u_j-u_{j-1})\\
	&=(he^{\xi}+k)\cdot(u_j-u_{j-1})\\
	&\leq 0,
	\end{align*}
	where the function $\xi:X\to \IR$ with  $u_{j-1}\geq \xi\geq u_j$ is given by the mean value theorem.
	Here, the last inequality follows from the definition of $k$. Note that $u_1\geq \xi$ by the induction hypothesis.
	Thus, the maximum principle, Lemma~\ref{MaxPrin}, yields $u_{j+1}\geq u_j$. 
	
	Next we show $u_j\geq u_0$ for every $j\geq1$.
	The case $j=1$ follows by the assumption  $u_1\geq u_0$.
	Suppose  $u_j\geq u_0$ for some $j\geq1$. We calculate
	\begin{align*}(\NeuLap+k) (u_{j+1}-u_0)&=f_{u_{j}}+ku_{j}-\NeuLap u_0-ku_0\\&\geq h\cdot(e^{u_{j}}-e^{u_0})+k\cdot(u_{j}-u_0)\\&=(he^{\xi}+k)\cdot(u_j-u_{0})\\&\geq 0, \end{align*} 
	where the function $\xi:X\to \IR$ with  $u_{j}\geq \xi\geq u_0$ is given by the mean value theorem.
	Thus, using the maximimum principle we infer $u_0\leq u_j$. 
	
	Hence, the sequence $u_j$ is pointwise monotonically decreasing  and, thus, there is a pointwise limit $u:X\to\IR$ with $u_0\leq u\leq u_1$. Using the finiteness of $m$ and Lebesgue's dominated
	convergence theorem, we infer $u\in\ltwo$ and $\|u_j-u\|_2\to 0$, $ j\to\infty $. 
	On the other hand we have $ u_{j}\in D(\NeuLap) $ by $ D(\NeuLap+k) =D(\NeuLap)$, Lemma~\ref{lemma_L+k}. Hence, $ (\NeuLap+k) u_j=f_{u_{j-1}}+ku_{j-1} $ yields

 \[\NeuLap u_j=f_{u_{j-1}}+ku_{j-1}-k u_j\to f_u\] in $\ltwo$ by Lebesgue's  dominated
 convergence theorem.
	The closedness of the operator $\NeuLap$ yields $u\in D(\NeuLap)$. By definition of $ f_{u}=he^{u}-c $ we conclude the equality
	$$ \NeuLap u=he^{u}-c. $$ 
\end{proof}

By the lemma above it suffices to present a lower and an upper solution $ u_{0} $ and $ u_{1} $ with $ u_{0}\leq u_{1} $ to prove Theorem~\ref{t:c<0}~(2). The next lemma shows the existence of a lower solutions in general and upper solutions for  $ c <0$ sufficiently large.

\begin{lemma}\label{lemma_upper_lower_solution}
	Let $b$ be a connected, canonically compactifiable graph over $(X,m)$, $ m(X)<\infty $ and let $h\in \ell^2(X,m)$ be such that $ \overline h <0 $.
	\begin{itemize}
		\item [(a)] For all $c<0$, there is a lower solution  $u_{0}\in  D(\NeuLap)$.
		\item [(b)]  There is $ c_{-}(h) <0$ such that there exists  an upper solution in $u_{1}\in D(\NeuLap)$ for $ c_{-}(h)<c<0 $. If furthermore $ h\leq 0 $, then $ c_{-}(h) =-\infty$.
	\end{itemize}
	Furthermore, $ u_{0} $ and $ u_{1} $ can be chosen such that $ u_{0}\leq u_{1} $.
	\end{lemma}
\begin{proof}
	Recall that the operator $ \NeuLap $ is bijective on the $ \ell^{2} $ functions which are orthogonal to the constants, Proposition~\ref{KerLN}. We  denote this inverse with slight abuse of notation by $ {\NeuLap}^{-1} $. Then $\NeuLap^{-1}$  maps the orthogonal complement of the constants into $ D(\NeuLap)\subseteq D(\NeuDir)\subseteq\ell^{\infty}(X)$ and we
	 define
	\begin{align*}
	v_{\alpha,\beta}=-\alpha\NeuLap^{-1}(h_{-}-\overline h_{-}) - \beta
	\end{align*}
for $\alpha,\beta\in\IR$ where $ f_{\pm} $ denotes the positive and negative part of a function  $  f$, i.e., 
$ f_{\pm}=(\pm f)\vee 0.$  Hence, $ v_{\alpha,\beta}\in D(\NeuLap)\subseteq \ell^{\infty} (X)$ and $ v_{\alpha,\beta}\to-\infty  $ uniformly for $\beta\to\infty$  and $\alpha\in\R $ fixed. We estimate $ -h \leq h_{-}$ to get
	\begin{align*}
	\NeuLap v_{\alpha,\beta}-he ^{v_{\alpha,\beta}}+c&\leq-\alpha h_{-}+\alpha\overline h_{-}+h_{-}e^{v_{\alpha,\beta}}+c\\&=(e^{v_{\alpha,\beta}}-\alpha)h_{-}+\alpha\overline h_{-}+c.
	\end{align*}
	Choosing $0\leq \alpha_{0}\leq |c|/\overline h_{-} $ and $ \beta _{0}$ sufficiently large we obtain the statement (a) with $ u_{0}=v_{\al_{0},\beta_{0}} $.
	
To show (b) let $ N $  be such that $ h_{N}=h\vee(- N )$ still satisfies $ \overline h_{N}<0 $. Furthermore,
 let $ C:=\|\NeuLap^{-1}(h_{N}-\overline h_{N}) \|_{\infty}$  and define for
$ \alpha\ge0 $
	\begin{align*}
	v_{\alpha}=\alpha(\NeuLap^{-1}(h_{N}-\overline h_{N})- C ).
	\end{align*}
	We observe that $-2\alpha C\le v_{\alpha} \le 0$. 
	We calculate 
	\begin{align*}
	\NeuLap (v_{\alpha}+\log \alpha)-he ^{v_{\alpha}+\log\al}&\geq
	\NeuLap (v_{\alpha}+\log \alpha)-h_{N}e ^{v_{\alpha}+\log\al}\\
	&=(\alpha- e^{v_{\alpha}+\log \alpha})h_{N}-\alpha\overline h_{N}\\
	&=\alpha(1- e^{v_{\alpha}})h_{N}-\alpha\overline h_{N}.	
	\end{align*}
	Next, we apply  the mean value theorem and infer that there is a function $ \xi:X\to [-2\alpha C,0 ] $ with $1-e^{v_\alpha}=-e^\xi v_\alpha$, use that $0\geq v_{\alpha}\geq -2\alpha C $ and denote by $ h_{N,-} $ the negative part of $ h_{N} $ to deduce
\begin{align*}
\ldots&=-\alpha e^{\xi} v_{\alpha} h_{N}-\alpha\overline h_{N}\\&\geq \alpha e^{\xi} v_{\alpha} h_{N,-}-\alpha\overline {h_N}\\
		&\geq -2\alpha^{2}e^{2\alpha C}CN-\alpha \overline{h_N}.
		\end{align*}
		Since $ \overline{h}_{N} <0$ there is $ \alpha_{0} $ such that
		\begin{align*}
		 C_{1}:=-2\alpha_{0}^{2}e^{2\alpha_{0} C}CN-\alpha_{0} \overline{h}_{N}>0
		\end{align*} 
		Hence, for $0> c_{-}(h)> -C_{1} $ we infer that $ u_{1}=v_{\al_{0}}+\log \al_{0} $ is an upper solution for $ 0>c>c_{-}(h) $. This is the first statement of (b). 
		
		Now, assume $ h\leq0 $.  For
$ \alpha\ge0$ define
	\begin{align*}
	v_{\alpha}=\alpha(\NeuLap^{-1}(h-\overline h)-C )
	\end{align*}
	with $C:=\|\NeuLap^{-1}(h-\overline h) \|_{\infty}$. Then, $ v_\al\leq 0$ and we deduce			
	\begin{align*}
			\NeuLap( v_{\alpha}+\log\al)-he ^{v_{\alpha}+\log\al}=\al (1-e^{v_{\al}})h-\al\overline h\geq-\al \overline{h}
			\end{align*}
			since  $ \alpha >0 $, $ v_{\al}\leq0 $, $ h\leq 0 $.			So, for given $ c<0 $  let $ \al_{c} \ge 0$ be such that $ -\al_{c}\overline{h}\geq -c $. Then, $  u_{1}=v_{\al_{c}}+\log\al_{c}$ is an upper solution.

			To see that we can choose $ u_{0} \leq u_{	1}$ we observe that $ u_{0},u_{1}\in D(\NeuLap)\subseteq \ell^{\infty}(X) $ and $ \beta_{0} $ in the definition of $ u_{0} $ can be chosen arbitrary large.
\end{proof}

\begin{proof}[Proof of Theorem~\ref{t:c<0}]
	$(1)$: Let $u$ be a solution. Then we have 
	\begin{align*}
	\langle h,1\rangle=
	\langle
	 he^{u},e^{-u}\rangle&=	\langle\NeuLap u,e^{-u}\rangle+c\langle 1,e^{-u}\rangle=\NeuDir(u,e^{-u})+c\langle 1,e^{-u}\rangle<0,
	\end{align*}
	where the last inequality follows since $ t\mapsto e^{-t} $ is monotone decreasing and the definition of $ \NeuDir $.
	
	$(2)$: This follows from  Lemma~\ref{lemma_upper_lower_implies_solution} and Lemma~\ref{lemma_upper_lower_solution}.
\end{proof}

\bibliography{literature}{}
\bibliographystyle{alpha}
\end{document}